\documentclass[letterpaper,onecolumn]{IEEEtran}
\renewcommand{\baselinestretch}{1.5}

\IEEEoverridecommandlockouts%
\usepackage[T1]{fontenc}
\usepackage{pdfsync}
\usepackage{verbatim}

\usepackage{amsmath,amssymb}

\usepackage{url}

\usepackage[font={small}]{caption}

\usepackage{subcaption}
\usepackage{color,tikz}
\definecolor{green}{rgb}{0,0.4,0}
\newcommand{\R}{\mathbb{R}}

\newcommand{\cB}{\mathcal{B}}
\newcommand{\cC}{\mathcal{C}}
\newcommand{\cD}{\mathcal{D}}
\newcommand{\cE}{\mathcal{E}}
\newcommand{\cG}{\mathcal{G}}
\newcommand{\cI}{\mathcal{I}}
\newcommand{\cJ}{\mathcal{J}}

\newcommand{\cN}{\mathcal{N}}
\newcommand{\cO}{\mathcal{O}}
\newcommand{\cP}{\mathcal{P}}
\newcommand{\cS}{\mathcal{S}}
\newcommand{\cU}{\mathcal{U}}
\newcommand{\cV}{\mathcal{V}}
\newcommand{\cW}{\mathcal{W}}
\newcommand{\cX}{\mathcal{X}}
\newcommand{\cY}{\mathcal{Y}}

\newcommand{\bA}{\bar{A}}
\newcommand{\bB}{\bar{B}}
\newcommand{\bC}{\bar{C}}

\newcommand{\bK}{\bar{K}}
\newcommand{\bM}{\bar{M}}

\newtheorem{definition}{Definition}
\newtheorem{theorem}{Theorem}

\newtheorem{corollary}{Corollary}

\newtheorem{lemma}{Lemma}

\newtheorem{remark}{Remark}

\title{\LARGE\textbf{Static Output Feedback:}\\ \textbf{On Essential Feasible Information Patterns}}

\author{%
    J. Frederico~Carvalho~$^{\sharp,\ast}$\quad %
    S\'ergio~Pequito~$^{\dagger,\ast}$\quad %
    A.~Pedro~Aguiar~$^{\diamond}$\quad \\ 
    Soummya~Kar~$^{\ddagger}$\quad %
    George~J.~Pappas~$^{\dagger}$
    \thanks{%
   This work was supported in part by the TerraSwarm Research Center, one of six centers supported by the STARnet phase of the Focus Center Research Program (FCRP) a Semiconductor Research Corporation program sponsored by MARCO and DARPA, and the NSF ECCS-1306128 grant.}%
    \thanks{${}^{\ast}$%
    The first two authors made equal contributions to the research and writing of this paper.}%
    \thanks{${}^{\dagger}$%
    Department of Electrical and Systems Engineering, School of Engineering and Applied Science, University of Pennsylvania}%
    \thanks{${}^{\ddagger}$%
    Department of Electrical and Computer Engineering, Carnegie Mellon University, Pittsburgh, PA 15213.}%
    \thanks{${}^{\diamond}$%
    Department of Electrical and Computer Engineering, Faculty of Engineering, University of Porto (FEUP), Portugal.}%
    \thanks{${}^{\sharp}$%
    Computer Science and Communication School, Royal Technical Institute (KTH), Stockholm, Sweden.}
}

\begin{document}

\maketitle

\begin{abstract}
In this paper, for linear time-invariant plants, where a collection of possible inputs and outputs are known a priori, we address the problem of determining the communication between outputs and inputs, i.e., information patterns, such that desired control objectives  of the closed-loop system (for instance, stabilizability) through static output feedback may be ensured.

We address this problem in the structural system theoretic context. To this end, given a specified structural pattern (locations of zeros/non-zeros) of the plant matrices, we introduce the concept of  \emph{essential} information patterns, i.e., communication patterns between outputs and inputs that satisfy the following conditions: $(i)$ ensure arbitrary spectrum assignment of the closed-loop system, using static output feedback constrained  to the information pattern,   for almost all possible plant instances with the specified structural pattern; and $(ii)$ any communication failure precludes the resulting information pattern from attaining the pole placement objective in $(i)$.

Subsequently, we  study the problem of  determining essential information patterns.
First, we provide several necessary and sufficient conditions to verify whether a specified information pattern is essential or not. Further, we show that such conditions can be verified by resorting to algorithms with polynomial complexity (in the dimensions of the state, input and output). Although such verification can be performed efficiently, it is shown that the problem of determining essential information patterns is in general NP-hard.  The main results of the paper are illustrated through examples.
\end{abstract}
\vspace{-0.2cm}

\section{Introduction}
\vspace{-0.1cm}

Real world systems are often too complex to be tackled by the classical paradigm of centralized decision-making. These systems include multi-agent  networks, infrastructure systems such as the electric power grid, process control and manufacturing systems, just to name a few~\cite{Sandel_etall:1978,siljak2007large,decentralizedSiljak}.
Furthermore, due to the distributed nature of the sensing-actuation capabilities of the aforementioned systems, there exist more often than not, a multitude of decision-makers. Therefore, the crafted communication structures need to take into account that only partial data may be accessed by the decision-makers, while guaranteeing that a desired closed-loop control performance is achievable.  Some groundbreaking work in the understanding of necessary and sufficient conditions to ensure arbitrary spectrum placement of  closed-loop systems constrained to specified information patterns can be found in~\cite{Sandel_etall:1978,Davison,morse,gong}. In recent years, research in decentralized control has seen several advances that contributed to a renewed interest in the field~\cite{Mahajan,Bakule08,BakuleP12,RotkowitzMTNS14,RotkowitzCDC14,Yukseldecentralizedcomputation,Langbort2009}.

In order to contribute to the understanding of how to instrument the communication between decision-makers, we specifically focus on addressing the following  question:

\begin{itemize}
\item[\textbf{Q}] What is the \emph{essential information pattern}, i.e., which sensors need to supply data to which actuators, such that desired control objectives (for instance, stabilizability) may be ensured, and any communication failure renders these objectives unattainable?
\end{itemize}

Hereafter, the desired control objective consists in ensuring that the spectrum of the closed-loop system, using static output feedback subject to the information pattern (IP), can be arbitrarily chosen. The above question has a direct consequence in the study of the  resilience of a closed-loop system with respect to communication failures or general changes in the IP. Furthermore, it  provides unique insights on the design of robust communication structures between decision-makers.

Towards this goal, we resort to a structural system theoretic framework~\cite{dionSurvey}, where equivalence classes of system instances with a specified  zero/non-zero pattern of the plant matrices are studied. A \emph{feasible} IP in this context corresponds to a communication pattern between outputs and inputs that ensures, for almost all plant instances with the specified structural pattern, arbitrary spectrum assignment of the corresponding closed-loop system is achievable through static output feedback constrained to the IP.
For a specified structural system, conditions that verify whether  an IP is feasible or not were provided in~\cite{sezerFixedModes,Sezer1981641,Pichai3}, and used for the design of feasible IPs in~\cite{PequitoJournal,PequitoNecsys13,PequitoACC15,Pajic}. Further, in~\cite{PequitoACC15} it was shown that the problem of determining the minimum cost feasible IPs, given a system plant and an input/output configuration,  is NP-hard. In particular, when we restrict the problem to one with a uniform cost on the communication links, we obtain the problem of determining sparsest feasible IPs, which is also known to be NP-hard. Nonetheless, in~\cite{PequitoACC15} it was also shown that if the dynamics matrix is irreducible then the minimum cost feasible IPs can be determined by resorting to algorithms with polynomial complexity (in the dimensions of the state, input and output). In this paper, one of the goals is to provide insights on how the  conditions required to ensure feasibility contribute to the hardness of the problem of determining the sparsest feasible IPs, as well as the essential IPs.  Finally, we notice that the essential IPs provide new insights on how to obtain  solutions to the general minimum cost  IP design, in the same lines as~\cite{PequitoACC15,Trave,optimumfeedbackpattern}, and to obtain resilient properties of the closed-loop system with respect to a given IP.

Some meaningful advances were recently achieved in terms of determining the numerical gains to achieve desired closed-loop system performance, given the existence of feasible IPs, and  accomplished by using convex optimization tools~\cite{Rotkowitz}. More precisely, gains associated to so-called \emph{quadratic invariant}~(QI) IPs can be determined using convex optimization tools~\cite{Lessard,LessardAlgControlDec}; see also~\cite{Mahajan} for a review about  recent developments. Recently, these results were also extended to enforcing sparsity in the IP~\cite{ChandrasekaranRegforDesign}, as well as part of a co-design problem with input and output selection~\cite{MatniCoDesign} to ensure that the IP associated with the communication between different decision-makers is QI, allowing for the design of the corresponding gain by  resorting to convex optimization tools. Alternatively, by also resorting to convex optimization tools, several other sparsity-promoting design of IPs were suggested in~\cite{Zecevic,linfarjovTAC13admm}.

The main contributions of this paper are threefold: $(i)$ we provide several necessary and sufficient conditions to verify whether an information pattern is an essential pattern or not; $(ii)$ we show that the problem of determining essential feasible information patterns is NP-hard; and $(iii)$ we provide a set of strategies that, given  an essential information pattern, enable us to determine a collection of essential information patterns.

The rest of the paper is organized as follows. In Section~\ref{probStatement}, we provide the formal problem statement.  Section~\ref{prelim} reviews some concepts and introduces results in structural systems theory and computational complexity. Subsequently, in Section~\ref{mainresults}, we present the main technical results, followed by an illustrative example in Section~\ref{illustrativeexample}. Conclusions and discussions on further research are presented in Section~\ref{conclusions}.

\section{Problem statement}
\label{probStatement}
Consider a linear time-invariant (LTI) system described by
 \begin{equation}%
         \dot x(t) = A x(t) + B u(t), \quad y(t) =C x(t),
     \label{eq:lti}%
 \end{equation}%
where $x\in\R^n$, $u\in\R^p$ and $y\in\R^m$  are the state, input and output vectors, respectively.

In the sequel, we identify the system in~\eqref{eq:lti} with the tuple $(A,B,C)$. In many real world scenarios with, specially in large-scale systems, it is often the case that the exact values of the non-zero parameters of the plant matrices  are unknown, or that these may change over time. To circumvent this problem, in this paper we adopt the framework of structural systems~\cite{dionSurvey}. To this end, we let $\bA \in {\{0,1\}}^{n \times n}$, $\bB\in{\{0,1\}}^{n\times p}$ and $\bC\in{\{0,1\}}^{m\times n}$ be the binary matrices that represent the structural patterns (location of zeros and non-zeros) of $A, B$, and $C$, respectively. We then focus on properties of all systems, where the plant matrices have these sparsity patterns $(\bA,\bB,\bC)$ which we refer to as a \emph{structural system}.

We thus consider the design of \emph{information patterns}  $\bar K\in{\{0,1\}}^{p\times m}$, where  $\bK_{i,j}=1$ if the measurements from output $j$ are available by actuator $i$, and zero otherwise. These patterns induce a sparsity on the static output feedback gains $K\in\R^{p\times m}$, with $u(t)=Ky(t)$ in~\eqref{eq:lti}, that leads to a closed-loop system, which we refer to as $(A,B,K,C)$.

In this setting, a \emph{feasible} information pattern $\bK$ is one that ensures  that the  closed-loop system has no \emph{fixed modes}~\cite{Davison}. To this end, given a matrix $\bM\in{\{0,1\}}^{m\times n}$, denote by $[\bM]$, the set $\{M \in \R^{m\times n}: M_{i,j}=0 \text{ \emph{if} } \bM_{ij}=0\}$. The set of fixed modes of the closed-loop system~\eqref{eq:lti} with respect to (w.r.t.) the information pattern $\bK$ is given by $\sigma_{\bK}(A,B,C)=\bigcap_{K\in [\bK]}\sigma(A+BKC)$, where $\sigma(M)$ denotes the set of eigenvalues of the matrix $M$, further if $\sigma_{\bK}\subset\cW$,  for a non-empty open set $\cW\subset \mathbb{C}$, which is symmetric about the real axis, then there exists a gain $K\in [\bK]$ such that all the eigenvalues of the closed-loop system matrix $A + BKC$ are in $\cW$ (see~\cite{Davison}).

Hereafter, we consider the notion of structural fixed modes introduced in~\cite{Papadimitriou84asimple}, which, essentially, are the fixed modes that arise from the structural pattern of a system. More concretely, a structural LTI system $(\bA,\bB,\bC)$ is said to have structurally fixed modes (SFMs) w.r.t.\  an information pattern $\bK$, which we refer to as $(\bar A,\bar B,\bar K,\bar C)$, if for all $A\in [\bA]$, $B\in [\bB]$, $C\in [\bC]$, we have $\sigma_{\bK}(A,B,C)\neq \emptyset$.

Conversely, a structural system $(\bA,\bB,\bK,\bC)$ has no structurally fixed modes, if there exists at least one instantiation $A\in [\bA]$, $B\in [\bB]$, $C\in [\bC]$ which has no fixed modes (i.e., $\sigma_{\bK}(A,B,C)=\emptyset$). In this latter case, it may be shown (see~\cite{sezerFixedModes}) that almost all closed loop systems in the sparsity class $(\bA,\bB,\bK,\bC)$ have no fixed modes, and, hence, allow pole-placement arbitrarily close to any pre-specified (symmetrical about the real axis) set of eigenvalues by a static output feedback with the sparsity of~$\bK$.

In summary, we choose the non-existence of SFMs as our design criterion for the structure $\bK$. Further, we say that $\bM'$ is a (strict) sub-pattern of $\bM$, which we write $\bM' < \bM$ if $[\bM'] \subsetneq [\bM]$. Therefore, we  aim at computing the \emph{essential} feasible information patterns, i.e., information patterns $\bK$ such that any $\bK'$ with $\bK'< \bK$, implies that $(\bA,\bB,\bK',\bC)$ has structurally fixed modes, i.e., is unfeasible. Formally, we explore the following problem:

\vspace{2mm}
\subsubsection*{$\cP_1$}
Let  $\bA \in {\{0,1\}}^{n\times n}$, $\bB\in {\{0,1\}}^{n\times p}$ and $\bC \in {\{0,1\}}^{m\times n}$ correspond to the dynamics, input and output matrices, respectively. Determine the essential feasible information patterns $\bK$, that is, $\bK$ such that $(\bA,\bB,\bK,\bC)$ has no structurally fixed modes and there exists no $\bK'$ such that $\bK'< \bK$ and $(\bA,\bB,\bK',\bC)$ has no structurally fixed modes.
\vspace{-0.5cm}

\hfill$\diamond$

Note that a characterization of the essential information feasible patterns yields a characterization of all feasible information patterns, since any feasible information pattern $\bK$ must have $\bK'\le\bK$ for some essential feasible information pattern $\bK'$. Further note that there is a particularly interesting class of essential feasible information patterns which are the sparsest feasible information patterns correspond  to the feasible information patterns that have the lowest  number of non-zero entries.
\vspace{-0.1cm}

\section{Preliminaries and terminology}\label{prelim}
\vspace{-0.1cm}

In this section, we review some basic concepts of structural systems and graph theory, followed by concepts of computational complexity. In addition, we introduce terminology that will be employed throughout the rest of paper.

Consider a linear time-invariant (LTI) system~\eqref{eq:lti}. In order to perform structural analysis efficiently, it is customary to associate to~\eqref{eq:lti} a directed graph (digraph) $\cD=(\cV,\cE)$,  in which $\cV$ denotes the set of \textit{vertices} and $\cE\subseteq\cV\times\cV$ the set of \textit{edges}, where $(v_j,v_i)$ represents an edge from the vertex $v_j$ to vertex $v_i$. To this end, let $\bA\in{\{0,1\}}^{n\times n}$, $\bB\in{\{0,1\}}^{n\times p}$ and $\bC\in{\{0,1\}}^{m\times n}$ be binary matrices that represent the sparsity patterns of $A$, $B$ and $C$ respectively. Denote by $\cX=\{x_1,\ldots,x_n\}$, $\cU=\{u_1,\ldots,u_p\}$ and $\cY=\{y_1,\ldots,y_m\}$ the sets of state, input and output vertices, respectively. And by $\cE_{\cX,\cX}=\{(x_i,x_j):\ \bA_{ji}\neq 0\}$, $\cE_{\cU,\cX}=\{(u_j,x_i):\ \bB_{ij}\neq 0\}$, and $ \cE_{\cX,\cY}=\{(x_i,y_j):\ \bC_{ji}\neq 0\}$  the edges between the sets in subscript; further, given an \emph{information pattern} $\bK\in{\{0,1\}}^{p\times m}$, describing  output feedback in the inputs, we also have $\cE_{\cY,\cU}=\{(y_j,u_i):\ \bK_{ij}\neq 0\}$. In addition,  we  introduce \emph{state} digraph $\cD(\bA)=(\cX,\cE_{\cX,\cX})$, and the \emph{closed-loop system } digraph $\cD(\bA,\bB,\bK,\bC)=(\cX\cup \cU\cup \cY,\cE_{\cX,\cX}\cup \cE_{\cU,\cX}\cup \cE_{\cX,\cY} \cup \cE_{\cY,\cU})$.

A \emph{directed path} between the vertices $v_1$ and $v_k$ is a sequence of edges $\{(v_1,v_2),(v_2,v_3),\hdots,(v_{k-1},v_k)\}$. If all the vertices in a directed path are different, then the path is said to be an \emph{elementary path}. A \emph{cycle} is an elementary path from $v_1$ to $v_k$, with an edge from $v_k$ to $v_1$.

We also require the following graph-theoretic notions~\cite{Cormen}: A digraph is strongly connected if there exists a directed path between any two vertices. A \emph{strongly connected component} (SCC) is a maximal subgraph  $\cD_S=(\cV_S,\cE_S)$ of $\cD$, i.e., a graph comprising a set of vertices $\cV'\subseteq\cV$ and of edges $\cE'\subseteq\cE$, such that for every $u,v \in\cV_S$ there exist paths from $u$ to $v$ and is maximal with this property (i.e., any $\cG$ such that $\cD_S\subseteq\cG\subsetneq\cD$, is not strongly connected).

Since the SCCs of a digraph $\cD=(\cV,\cE)$ are uniquely determined, we can regard each SCC as a virtual node. By doing so we build a \textit{directed acyclic graph} (DAG), i.e., a directed graph with no cycles,   in which a directed edge exists between two virtual nodes  representing two SCCs \emph{if and only if} there exists an edge between two vertices in the corresponding SCCs in the original digraph.  We call this, the DAG representation of the graph, which can be computed efficiently in $\cO(|\cV|+|\cE|)$~\cite{Cormen}. We can further classify the SCCs with respect to the existence of incoming and/or outgoing edges as follows.

\begin{definition}[\cite{PequitoJournal}]\label{linkedSCC}
    An SCC is said to be linked if it has at least one incoming or outgoing edge from another SCC\@. In particular, an SCC is \textit{non-top linked} if it has no incoming edges  from  another SCC, and \textit{non-bottom linked} if it has no outgoing edges to another SCC\@.
\hfill $\diamond$
\end{definition}

For any digraph $\cD = (\cV,\cE)$ and any two vertex sets $\cS_{1}, \cS_{2}\subset \cV$ we define the \textit{bipartite graph} $\cB(\cS_1,\cS_2,\cE_{\cS_1,\cS_2})$ whose vertex set is given by $\cS_{1}\cup \cS_{2}$ and the edge set $\cE_{\cS_1,\cS_2}=\cE\cap(\cS_1\times\cS_2)$. We call the bipartite graph $\cB(\cV,\cV,\cE)$ the bipartite graph associated with $\cD(\cV,\cE)$. In the sequel we will make heavy use of the \emph{state bipartite graph} $\cB(\bA)\equiv \cB(\cX,\cX,\cE_{\cX,\cX})$, which is the bipartite graph associated with the state digraph $\cD(\bA)=(\cX,\cE_{\cX,\cX})$.

Given $\cB(\cS_1,\cS_2,\cE_{\cS_1,\cS_2})$, a matching $M$ corresponds to a subset of edges in $\cE_{\cS_1,\cS_2}$ so that no two edges have a vertex in common, (i.e.,~given edges $e=(s_1,s_2)$ and $e'=(s_1',s_2')$ with $s_1,s_1' \in \cS_1$ and $s_2,s_2'\in \cS_2$, $e, e' \in M$ only if $s_1\neq s_1'$ and $s_2\neq s_2'$).
A maximum matching $M^{\ast}$ is a matching $M$ that has the largest number of edges among all possible matchings.

In addition, given two binary matrices $P$ and  $P'$, we define their sum $P + P'$, where we replace the binary sum with the entrywise \emph{or} operation.

We call the vertices in $\cS_1$ and $\cS_2$ belonging to an edge in $M^\ast$, the \textit{matched vertices} w.r.t.\ $M^\ast$, and  \textit{unmatched vertices} otherwise. 
For ease of referencing, in the sequel, the term \emph{right-unmatched vertices} associated with the matching $M$ of $\mathcal B(\mathcal S_1,\mathcal S_2,\mathcal E_{\mathcal S_1,\mathcal S_2})$ (not necessarily maximum), will refer to those vertices in $\cS_{2}$ that do not belong to a matching edge in $M$, dually a vertex from $\cS_1$ that does not belong to an edge in $M$ is called a \emph{left-unmatched vertex}.

The following result translates a maximum matching of the state bipartite graph representation into the state digraph.

\begin{lemma}[Maximum Matching Decomposition~\cite{PequitoJournal}]\label{matMatPathCycle}
    Consider the digraph $\cD(\bA)=(\cX,\cE_{\cX,\cX})$ and let $M^*$ be a maximum matching associated with the bipartite graph $\cB(\cX,\cX,\cE_{\cX,\cX})$. Then, the digraph $\cD=(\cX, M^*)$  comprises a disjoint union of cycles and elementary paths, from the right-unmatched vertices to the left-unmatched vertices of $M^{\ast}$, that span $\cD(\bA)$ (by definition an isolated vertex is regarded as an elementary path with no edges). Moreover, such a decomposition is \emph{minimal}, in the sense that no other spanning subgraph decomposition of $\cD(\bA)$ into elementary paths and cycles contains strictly fewer elementary paths.
    \hfill $\diamond$
\end{lemma}

Now, concerning the notion of structural fixed modes introduced in Section~\ref{probStatement} for the formulation of $\cP_1$, we will make heavy use of the following graph-theoretic conditions that ensure the absence of structurally fixed modes.

\begin{theorem}[\hspace{-0.02cm}\cite{Pichai3}]
    The structural system $(\bA,\bB,\bC)$ associated with~\eqref{eq:lti} has no structurally fixed modes w.r.t.\  an information pattern $\bK$, if and only if both the following conditions hold:

 $(a)$ in $\cD(\bA,\bB,\bK,\bC)=(\cX\cup\cU\cup \cY,\cE_{\cX,\cX}\cup\cE_{\cX,\cY}\cup\cE_{\cU,\cX}\cup\cE_{\cY,\cU})$, each state vertex $x\in \cX$ is contained in an SCC which includes an edge of $\cE_{\cY,\cU}$;

 $(b)$ there exists a finite disjoint union of cycles $\mathcal{C}_k=(\mathcal{V}_k,\mathcal{E}_k)$ (subgraph of $\mathcal{D}(\bA,\bB,\bK,\bC)$) with $k\in\mathbb{N}$ such that $\mathcal{X}\subset \bigcup_{j=1}^{k}  \cV_j $.       \hfill$\diamond$

\label{noStrucFixedModes}
\end{theorem}

Finally, a (computational) problem is said to be \emph{reducible in polynomial time} to another if there exists a procedure transforming a solution of the former in one of the latter resorting to a number of elementary operations which is bounded by a polynomial on the size of its inputs. Such reductions are useful in determining the  complexity class~\cite{Garey:1979:CIG:578533} a problem belongs to. For instance, recall that a decision problem $\cP$ in NP (i.e., the class of problems for which a solution can be verified in polynomial time) is said to be NP-complete if all other decision problems in NP can be polynomially reduced to $\cP$~\cite{Garey:1979:CIG:578533}. The set of NP-complete problems is referred to as the NP-complete class.

The optimization problems, whose associated decision problems are NP-complete are called NP-hard, and they form the NP-hard class. A typical result used to show the computational complexity of a problem is given next.

\begin{lemma}[\cite{Garey:1979:CIG:578533}]
    If a problem $\cP_A$ is NP-hard, $\cP_B$ is in NP and $\cP_A$ is reducible in polynomial time to $\cP_B$, then $\mathcal P_B$ is NP-hard.
    \hfill $\diamond$
\label{NPcompRed}
\end{lemma}

Concretely, we will make use the \emph{decomposition problem}, an NP-hard problem that is formulated as follows~\cite{tarjan}:

\subsubsection*{Decomposition Problem}

Given a directed acyclic graph $\cD=(\cV,\cE)$. Determine (if possible) a partition of $\cV$ in two sets, say $\Gamma_1$ and $\Gamma_2$ such that:

$(i)$ there are no edges leading from $\Gamma_2$ to $\Gamma_1$, and

$(ii)$ if $v \in \Gamma_i$, for $i\in\{1,2\}$, there is a \emph{source-to-sink} path in $\cD$ passing through $v$ containing only vertices of $\Gamma_i$,
where a vertex is a \emph{source} if it has no incoming edges and a \emph{sink} if it has no outgoing edges.
\hfill $\diamond$

The decomposition problem has been shown to be a computationally hard, as stated in the following theorem:

\begin{theorem}[\cite{tarjan}]\label{decompNP}
    The problem of determining a partition as in the decomposition problem is NP-hard.
    \hfill $\diamond$
\end{theorem}

Further, to clarify the distinction between NP-hard and NP-complete problems, we remark that a decision problem associated to the decomposition problem can be described as follows:  for an arbitrary digraph $\cD$ is there a partition of the set of vertices into two disjoint non-empty subsets $\Gamma_1,\Gamma_2$ satisfying conditions $(i)$ and $(ii)$.
\vspace{-0.1cm}

\section{Main results}\label{mainresults}
\vspace{-0.1cm}

In this section, we present the main results of this paper. We begin by addressing the problem of characterizing the  essential information  patterns by providing a description of those satisfying  condition $(b)$ of Theorem~\ref{noStrucFixedModes}. To this end, Theorem~\ref{informationCondb} and Corollary~\ref{cor:informationCondb} summarize the obtained results. Secondly, in Theorem~\ref{correachpair}, we consider the satisfiability of condition $(a)$ of Theorem~\ref{noStrucFixedModes}. Next, in Theorem~\ref{thm:general-np-comp} we show that the problem of finding essential  feasible information  patterns is in general NP-hard, despite of the previous theorems hinting at efficient (polynomial in the dimension of the state) algorithms to find feedback patterns that satisfy conditions $(a)$ and $(b)$ of Theorem~\ref{noStrucFixedModes} separately. Finally,  we describe in Lemma~\ref{bisection} and Lemma~\ref{splitMethod}, methods by which we can obtain essential    information patterns from pre-existing essential  information patterns.

We thus begin by defining index- and sequential-pairing,  that will be used throughout the remainder of the paper in order to describe how the communication from a set of sensors $\cI$  to a set of actuators $\cJ$ should be setup.

\begin{definition}[Index-pairing]
    Given two sets of indices $\cI=\{i_1,\ldots,i_n\}$ and $\cJ=\{j_1,\ldots,j_k\}$, we define an index pairing $\langle \cI , \cJ \rangle$ as being a maximum matching of the bipartite graph $\cB(\cI,\cJ,\cI\times\cJ)$.
    \hfill $\diamond$
\end{definition}

\begin{definition}[Sequential-pairing]
    Consider  two sets of indices $\cI=\{i_1,\ldots,i_n\}$ and $\cJ=\{j_1,\ldots,j_n\}$, and  a maximum matching $M$ of the bipartite graph $\cB(\cJ,\cI,\cE_{\cJ,\cI})$, where  $\cE_{\cJ,\cI} \subseteq \cJ\times\cI$. We denote by $|\cI,\cJ\rangle_M$ a sequential-pairing induced by $M$, defined as follows:\vspace{-0.6cm}

    $$
        |\cI,\cJ\rangle_M=\Bigg(\bigcup_{l=2,\ldots,k}\{(i_l,j_{l-1})\}\Bigg) \cup \big\{(i_1,j_k)\big\},
    $$\vspace{-0.4cm}

\noindent  where  $(j_l,i_l)\in M$, for $l=1,\dots,k$.
    \hfill $\diamond$
\end{definition}
We note that the sequential-pairing consists in  the collection of edges such that $M\cup |\cI,\cJ\rangle_M$ forms  a cycle.

Now, we begin by providing  necessary and sufficient conditions to ensure Theorem~\ref{noStrucFixedModes}--$(b)$.

\begin{theorem}\label{informationCondb}
    Let $\cD(\bA)=(\cX,\cE_{\cX,\cX})$ be the state digraph and $\cB(\bA)$ the associated  state bipartite graph. In addition, let the input and output matrices be given by $\bB=\mathbb{I}_n$ and $\bC=\mathbb{I}_n$, respectively. The following statements are equivalent:

$(i)$ The digraph $\cD(\bA,\mathbb{I}_n,\bK, \mathbb{I}_n)$ satisfies Theorem~\ref{noStrucFixedModes}$-(b)$;

$(ii)$ There exists a matching $M$ of $\cB(\bA)$, with a set of right-unmatched vertices $\cU_R = \left\{ x_{j_R} : j_R\in \cJ_R \right\}$ and left-unmatched vertices $\cU_L = \left\{ x_{j_L} : j_L\in \cJ_L \right\}$,  where $\cJ_R$ and $\cJ_L$ correspond to the indices of right- and left-unmatched vertices, respectively, such that  $\bK_{j_R,j_L}=1$ for all $(j_R,j_L)$ in some index-pairing $\langle\cJ_R,\cJ_L\rangle$  and is zero otherwise. \hfill $\diamond$
\end{theorem}

\begin{proof}
    To see that $(i)$ implies $(ii)$, consider a collection of disjoint cycles $\cC$ of $\cD\equiv \cD(\bA,\mathbb{I}_n,\bK, \mathbb{I}_n)$ that contains all the state variables, as prescribed by Theorem~\ref{noStrucFixedModes}$-(b)$. Further, denote by $\cE_{\cX,\cX}^c\subset \cE_{\cX,\cX}$  the set of edges between state vertices that are used in the cycles comprising $\cC$. Now, set $M=\cE_{\cX,\cX}^c$, where it can be seen (recall Lemma~\ref{matMatPathCycle}) that $M$ is a matching of $\cB(\bA)$; so, let $\cU_R$ and $\cU_L$ be the associated sets of right- and left-unmatched vertices, respectively. Therefore, because $\cC$ is a decomposition of $\cD(\bA,\mathbb{I}_n,\bK,\mathbb{I}_n)$ in cycles, it follows that there exists a set of inputs $\mathbb{I}_n^{\cJ_R}$ and outputs $\mathbb{I}_n^{\cJ_L}$ leading to edges from the inputs and outputs in $\cD$ labeled by $\cJ_R$ and $\cJ_L$, respectively. Further,  the inputs have outgoing edges into the state variables in $\cU_R$,  and the outputs with incoming edges from the state variables in $\cU_L$, respectively.  At last, because $\cC$ is a decomposition of $\cD$ in cycles,  there exists a feedback edge from the outputs labeled by $\cJ_L$ into the inputs labeled by $\cJ_R$; in other words, $\bK_{j_R,j_L}=1$  with $(j_R,j_L)\in \langle\cJ_R,\cJ_L\rangle$, and the result follows.

    On the other hand, $(ii)$ implies $(i)$ since, by Lemma~\ref{matMatPathCycle}, it follows that $M$ is such that $\mathcal D_M=(\mathcal X,M)$ spans $\cD(\bA)$ with a disjoint collection of cycles $\cC$ and paths $\cP$, where the paths start from vertices in $\cU_R$ and end in vertices in $\cU_L$. Further, $\cB(\bA)=(\cX,\cX,\cE_{\cX,\cX})$, for a given matching $M$ of $\cB(\bA)$,  we have $|\cU_R| = |\cU_L|$ and so $\langle\cJ_R,\cJ_L\rangle$ is a perfect matching, thus, since $\bK_{j_R,j_L}=1$ for $(j_R,j_L)\in \langle\cJ_R,\cJ_L\rangle$, it is possible to extend the paths in $\cP$ into cycles in $\cD(\bA,\mathbb I_n,\bK,\mathbb I_n)$ comprising input vertices, output vertices and feedback links. More precisely, those inputs with edges ending in $\cU_R$ and outputs with incoming edges starting in $\cU_L$. These cycles are disjoint from the cycles in $\cC$ (by Lemma~\ref{matMatPathCycle}), and together they form a family of cycles that contains all of the state vertices in $\cD(\bA,\mathbb{I}_n,\bK,\mathbb{I}_n)$.
\end{proof}

The sparsest information pattern satisfying Theorem~\ref{noStrucFixedModes}--$(b)$ can be obtained, as corollary to Theorem~\ref{informationCondb}, as described next.

\begin{corollary}\label{cor:informationCondb}
    Let $\cD(\bA)=(\cX,\cE_{\cX,\cX})$ be the state digraph and $\cB(\bA)$ the associated  state bipartite graph. In addition, let  the input and output matrices be  given by $\bB=\mathbb{I}_n$ and $\bC=\mathbb{I}_n$, respectively. The following statements are equivalent:

$(i)$  The digraph  $\cD(\bA,\mathbb{I}_n,\bK^*, \mathbb{I}_n)$, where  $\bK^*$ is a sparsest information pattern, satisfies Theorem~\ref{noStrucFixedModes}$-(b)$;

$(ii)$ There exists a maximum matching $M^*$ of $\cB(\bA)$, with set of right-unmatched vertices $\cU_R^*$ and left-unmatched vertices $\cU_L^*$, where $\cJ_R^*$ and $\cJ_L^*$  are the indices of the state variables comprised in each, respectively, such that  $\bK_{j_R,j_L}=1$ if $(j_R,j_L)\in \langle\cJ_R^*,\cJ_L^*\rangle$ and zero otherwise.
        \hfill $\diamond$
\end{corollary}

Now, we focus on the problem of determining the feedback patterns $\bK$ satisfying Theorem~\ref{noStrucFixedModes}--$(a)$. Towards this goal, consider the following auxiliary lemma.

\begin{lemma}\label{reachpair}
    Consider $(\bA,\bB,\bC)$ to be  a structural system with $\bB=\bC=\mathbb{I}_n$, and $\cD(\bA)$ the state digraph. Further, let the non-top linked SCCs of $\cD(\bA)$ be denoted by $\cN^T_1,\dots,\cN_{\beta_T}^T$, and the non-bottom linked SCCs by $\cN^B_1,\dots,\cN^B_{\beta_B}$. Then, provided $\beta_T\le\beta_B$ (resp. $\beta_B\le\beta_T$),  there exists a information pattern $\bK$ with $\beta_T$ (resp. $\beta_B$) nonzero entries, so that the digraph $\cD(\bA,\bB,\bK,\bC)$ has a unique non-top (resp.~non-bottom) linked SCC and $|\beta_B-\beta_T|$ non-bottom (resp.~non-top) linked SCCs.
    \vspace{-0.3cm}
    
    \hfill $\diamond$
\end{lemma}

\begin{proof}
    The proof follows by construction: assume that $\beta_T\le\beta_B$ and let $\cI = \left\{ 1,\dots,\beta_T \right\}$ and $\cJ=\left\{ 1,\dots,\beta_B \right\}$. Further, define the bipartite graph $\cB = \cB(\cI,\cJ,\cE_{\cI,\cJ})$ having an edge $(i,j)\in\cE_{\cI,\cJ}$ if there is a path from a vertex in $\cN_i^T$ to some vertex in $\cN_j^B$ (in which case we say that $\cN_i^T$ \emph{reaches} $\cN_j^B$). Given a maximum matching, $M^*$, of $\cB$, one of two things can happen; either $M^*$ has left-unmatched vertices, or not.

    In case there are no left-unmatched vertices in $\cB$ associated with $M^*$, let $\cJ'$ be the subset of $\cJ$ comprising those indices vertices belonging to the edges in $M^\ast$, and consider the sequential-pairing $|\cJ',\cI\rangle_{M^*}$. Further, let $\iota$ be a function that given an SCC of $\cD(\bA)$ returns the index of a single vertex in that SCC and define the information pattern $\bK'$ as $\bK'_{\iota(\cN^T_j),\iota(\cN^B_i)} = 1$ if the edge $(j,i)$ belongs to the sequential-pairing $|\cJ',\cI\rangle_{M^*}$, and $\bK'_{l,t}=0$ for all other pairs.

    Now, note that the resulting closed-loop system digraph $\cD(\bA,\bB,\bK',\bC)$ has a unique non-top linked SCC, since the edges induced by $\bK'$ ensure that there is a cycle going through all of the non-top linked SCCs of $\cD(\bA,\bB,\bC)$, and through $\beta_B-\beta_T$ of the non-bottom linked SCCs\@. Moreover, because no edges were added to the remaining $\beta_B-\beta_T$ non-bottom linked SCCs, these remain non-bottom linked in $\cD(\bA,\bB,\bK',\bC)$, and  the information pattern $\bK = \bK'$ satisfies the theorem statement.

     Finally, if there are left-unmatched vertices in $\cB$ associated with $M^*$, we must apply the procedure twice. To be more precise, we let $\cI'$ and $\cJ'$ be the subsets of $\cI$ and $\cJ$ respectively, that correspond to the indices of the vertices in the  edges in the maximum matching $M^*$. Then the edges induced by $\bK'$ in $\cD(\bA,\bB,\bK',\bC)$ give rise to a cycle going through the non-top and non-bottom linked SCCs indexed by $\cI'$ and $\cJ'$ respectively, and all these SCCs belong to the single SCC in $\cD(\bA,\bB,\bK',\bC)$.

     Now, since $\cB$ has left-unmatched vertices associated with $M^\ast$, it must be the case that there is some non-top linked SCC that only reaches non-bottom linked SCCs that were already matched by $M^\ast$. Note also that the non-top linked SCCs that were matched in $M^\ast$ reach all of the remaining non-bottom linked SCCs, otherwise there would be a non-bottom linked SCC that would be reachable from an unmatched non-top linked SCC making $M^\ast$ non-maximal. So, the non-top linked SCCs with indexes in $\cI\setminus\cI'$ reach all of the non-bottom linked SCCs in $\cD = \cD(\bA,\bB,\bK',\bC)$. Finally, note that the non-top linked and non-bottom linked SCCs of $\cD(\bA)$ indexed by $\cJ\setminus\cJ'$ and $\cI\setminus\cI'$ correspond to the non-top, and non-bottom linked SCCs of $\cD$, and that $\beta_B > \beta_T$ and that every non-top linked SCC of $\cD$ reaches every non-bottom linked SCC of $\cD$. This implies that a maximum matching $N^*$ of $\cB(\cI \setminus \cI',\cJ \setminus \cJ',\cE_{\cI \setminus \cI',\cJ \setminus \cJ'})$ has no left-unmatched vertices. So we perform a sequential-pairing $|\cJ\setminus\cJ',\cI\setminus\cI'\rangle_{N^*}$ and define $\bK''_{\iota(\cN^T_j),\iota(\cN^B_i)} = 1$ if the edge $(j,i)$ is in the pairing $|\cJ\setminus \cJ',\cI\setminus \cI'\rangle_{N^*}$ and $\bK''_{l,t}=0$ for any other pair of indices.

    By considering $\bK = \bK'+\bK''$, we obtain $\cD(\bA,\bB,\bK,\bC)$ satisfying the conclusions of the theorem.
\end{proof}

Subsequently, we have the following result regarding the information patterns $\bK$ satisfying Theorem~\ref{noStrucFixedModes}--$(a)$.

\begin{theorem}\label{correachpair}
    Consider $(\bA,\bB,\bC)$ a structural system with $\bB=\bC=\mathbb{I}_n$, and $\cD(\bA)$ be the state digraph.  Further, let the non-top linked SCCs of $\cD(\bA)$ be denoted by $\cN^T_1,\dots,\cN_{\beta_T}^T$, and the non-bottom linked SCCs by $\cN^B_1,\dots,\cN^B_{\beta_B}$. Then, provided $\beta_T\le\beta_B$ (respectively  $\beta_B\le\beta_T$) there exists a information pattern $\bK$ with $\beta_B$ (respectively  $\beta_T$) nonzero entries, so that $\cD(\bA,\bB,\bK,\bC)$ satisfies Theorem~\ref{noStrucFixedModes}--$(a)$. Further, this information pattern has the lowest number of non-zero entries in order to satisfy Theorem~\ref{noStrucFixedModes}--$(a)$.
    \hfill $\diamond$
\end{theorem}

\begin{proof}
    To construct the appropriate information pattern, assume that $\beta_T\le\beta_B$, then let $\bK'$ be the structural pattern provided by Lemma~\ref{reachpair}; this information pattern is such that $\cD(\bA,\bB,\bK',\bC)$ has a unique non-top linked SCC $\cN$, and non-bottom linked SCCs $\cN^B_1,\dots,\cN^B_{\beta_B-\beta_T}$. So, let $\bK''$ be such that $\bK''_{i_k,\iota(\cN^B_k)} = 1$, where $i_k$ is the index of an arbitrary state variable in $\cN$, and $\iota(\cN^B_k)$ provides the index of a state variable in the non-bottom linked SCC\@. So, let $\bK = \bK' + \bK''$; the resulting digraph $\cD(\bA,\bB,\bK,\bC)$, is strongly connected, and contains at least one feedback edge, thus satisfying Theorem~\ref{noStrucFixedModes}--$(a)$; furthermore, it comprises exactly $\max(\beta_B,\beta_T)$ feedback edges, thus making $\bK$ a sparsest information pattern satisfying Theorem~\ref{noStrucFixedModes}--$(a)$, since if less feedback edges were used, associated with the information pattern $\tilde{\bar{K}}$, then there would be a non-bottom linked SCC in $\cD(\bA,\bB,\bK,\bC)$ without feedback edges.
\end{proof}

\begin{remark}
   First, all the tools required to obtain Theorem~3, Corollary~1 and Theorem~4, can be implemented by resorting to algorithms with polynomial complexity (in the dimensions of the state, input and output), see Section~\ref{prelim}. Secondly, in all the results we gave so far, where condition $(a)$ in Theorem~\ref{noStrucFixedModes} is intended, we aimed to obtain a closed-loop digraph comprising a single SCC by producing information patterns with feedback links from outputs to inputs located in non-bottom and non-top linked SCCs, respectively. We chose this method, because the class of information patterns that satisfy condition $(a)$ of Theorem~\ref{noStrucFixedModes} is rather unruly, as illustrated  in Figure~\ref{fig:unruly-feedbacks}. However, for the class of state digraphs comprising the same number of non-top and non-bottom linked SCCs, all sparsest information patterns that satisfy the aforementioned condition, must follow the strategy previously employed.
    \hfill$\diamond$%
\label{rmk:nont-nonb-feedbacks}
\end{remark}
\begin{figure}[htpb]
      \centering
  \includegraphics[scale=0.3]{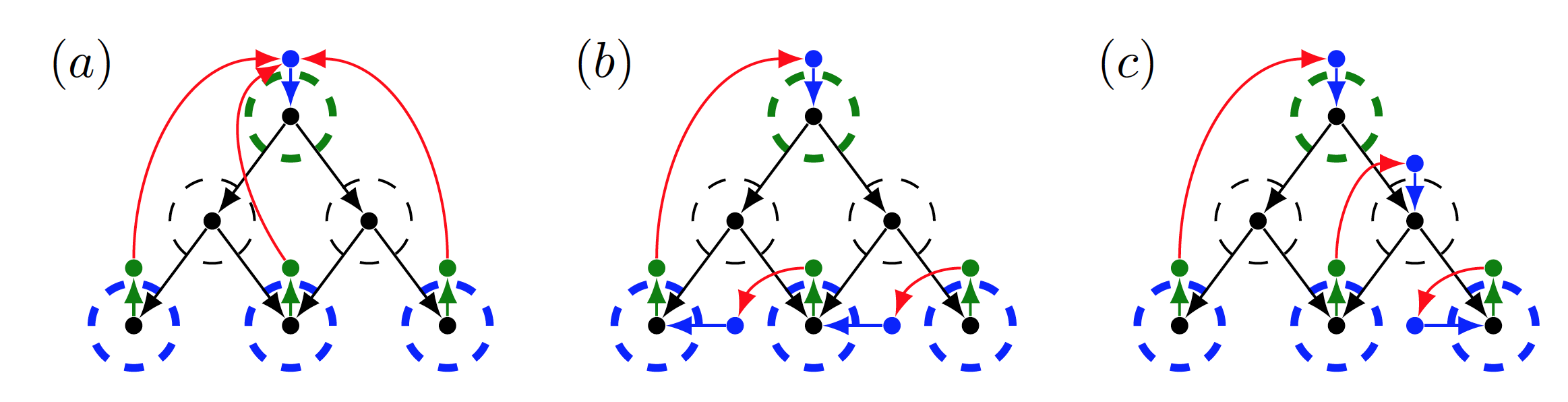}
    \caption{ In this figure, we present several information patterns that guarantee the satisfiability of Theorem~\ref{noStrucFixedModes}--$(a)$, where we represent the SCCs in dashed circles, the non-top linked SCCs are depicted in green and the non-bottom linked SCCs are depicted in blue, further the feedback edges are depicted in red. In $(a)$ we present the structural pattern obtained from Corollary~\ref{correachpair}, in $(b)$ we present an alternative information pattern that also produces a single SCC, and in $(c)$ we present a information pattern that comprises three SCCs while still satisfying Theorem~\ref{noStrucFixedModes}$-(a)$.}%
\label{fig:unruly-feedbacks}
\end{figure}

The main reason for pursuiting strategies as emphasized in Remark~1,  when we design $\bK$ to satisfy condition $(a)$, is closely related to the next theorem (Theorem~5) that is based in the next lemma.

\begin{lemma}\label{thm:still-np-comp}
 The problem of determining a sparsest information pattern $\bK^*$ such that $(\bA,\bB,\bK^*,\bC)$, with $\bB=\bC=\mathbb{I}_n$, satisfies condition $(a)$ of Theorem~\ref{noStrucFixedModes} and $\cD(\bA,\bB=\mathbb{I}_n,\bK^*,\bC=\mathbb{I}_n)$ has two SCCs comprising state variables is NP-hard.
\hfill $\diamond$
\end{lemma}

\begin{proof}
The proof follows by resorting to Lemma~\ref{NPcompRed}, where $\mathcal P_A$ is the (NP-hard) decomposition problem and $\mathcal P_B$ the problem of determining $\bar K$ such that condition $(a)$ of Theorem~\ref{noStrucFixedModes} is satisfied and $\cD(\bA,\bB=\mathbb{I}_n,\bK^*,\bC=\mathbb{I}_n)$ has two SCCs comprising state variables. Towards this goal let $\mathcal D=(\mathcal V=\{v_1,\ldots,v_n\},\mathcal E)$ in the decomposition problem to be the state digraph $\mathcal D(\bar A)=(\mathcal X,\mathcal E_{\mathcal X,\mathcal X})$ with $\mathcal X=\{1,\ldots,n\}$ and $\mathcal E_{\mathcal X,\mathcal X}=\{(x_i,x_j): \ (v_i,v_j)\in \mathcal E\}$, and $\bar B=\bar C=\mathbb{I}_{n}$; in addition, notice that the sources and sinks of $\mathcal D$ are the non-top and the non-bottom linked SCCs of $\mathcal D(\bar A)$, respectively. Further, such reduction has linear complexity; hence, $\mathcal P_A$ can be polynomially reduced to $\mathcal P_B$. Now, we show that this reduction is correct, i.e., a solution to $\mathcal P_B$ provides a solution to $\mathcal P_A$. Let $\mathcal N_1$ and $\mathcal N_2$ be the two SCCs containing state variables of $\cD(\bA,\bB=\mathbb{I}_n,\bK^*,\bC=\mathbb{I}_n)$  with a sparsest $\bK^*$.  As consequence of the constructions provided in Theorem~\ref{correachpair}, it follows that $\mathcal N_1$ and $\mathcal N_2$ only contain input vertices with outgoing edges into the state variables that are non-top linked SCCs of $\mathcal D(\bar A)$. Similarly, the $\mathcal N_1$ and $\mathcal N_2$ only contain output vertices with incoming edges from the state variables that are non-bottom linked SCCs of $\mathcal D(\bar A)$.
Therefore,  $\Gamma_1=\{v_i \in \mathcal V: \ x_i \in \mathcal X\cap \mathcal N_1 \}$ and $\Gamma_2=\{v_i \in \mathcal V: \ x_i \in \mathcal X\cap \mathcal N_1 \}$ are partitions of $\mathcal D$, and they satisfy the decomposition problem because $\mathcal D$ is DAG, which implies that there exists only directed edges from one of the partitions to the other. Finally, $\mathcal P_B$ is a NP problem, since the satisfaction of Theorem~\ref{noStrucFixedModes}-$(a)$ can be verified polynomially, by determining the different SCCs of $\cD(\bA,\bB=\mathbb{I}_n,\bK^*,\bC=\mathbb{I}_n)$ (see for instance~\cite{Cormen}) and verifying if only two SCCs contain state variables. Therefore, all conditions of Lemma~\ref{NPcompRed} are satisfied, which implies that $\mathcal P_B$ is NP-hard, and the result follows.
\end{proof}

In fact, noticing that the problem in Lemma~\ref{thm:still-np-comp} is an instance of a more general problem with arbitrary  input and  output matrices, and  the fact that the sparsest information patterns are essential information patterns, we obtain the following result.

\begin{theorem}\label{thm:general-np-comp}
 The problem of determining an  essential information pattern $\bK^*$ such that $(\bA,\bB,\bK^*,\bC)$ satisfies the  conditions in Theorem~\ref{noStrucFixedModes} is NP-hard.
\hfill $\diamond$
\end{theorem}

Motivated by Theorem~\ref{thm:general-np-comp}, and to partially address $\mathcal P_1$ in an efficient manner,  we have the following lemmas.

\begin{lemma}
    Let $\cD(\bA,\mathbb{I}_n,\bK,\mathbb{I}_n)$ satisfy Theorem~\ref{noStrucFixedModes}$-(a)$. Then, for some $i,j$ such that $\bK_{i,j} = 1$, let $\bK'$ be such that $\bK'_{i,j} = 0$, $\bK'_{i',j} = 1$ and $\bK'_{i,j'}=1$ for some $i',j'$ such that $x_{i'}$ and $x_{j'}$ belong to the same SCC, and $\bK'_{l,t} = \bK_{l,t}$ for every other pair of indices $l,t$. The resulting closed-loop system digraph $\cD(\bA,\mathbb{I}_n,\bK',\mathbb{I}_n)$ satisfies Theorem~\ref{noStrucFixedModes}$-(a)$.
    \hfill $\diamond$%
\label{bisection}
\end{lemma}

\begin{proof}
    Since $\cD(\bA,\mathbb{I}_n,\bK,\mathbb{I}_n)$ satisfies Theorem~\ref{noStrucFixedModes}$-(a)$, we have that every SCC has a feedback link on it. By considering $\cE'_{\cY,\cU}$, because $x_{i'},x_{j'}$ belong to the same SCC, it follows that the SCC that comprised the edge $\{(y_i,u_j)\}$ has now two feedback edges, i.e., $\{(y_i,u_{j'})\}$ and $\{(y_{i'},u_j)\}$.
\end{proof}

In the following, we denote $(x_1,x_2,x_3,\dots,x_k)$ to mean the path $\{(x_1,x_2),\dots,(x_{k-1},x_k)\}$; and, if $x_1=x_k$ we obtain  a cycle $(x_1,x_2,x_3,\dots,x_k,x_1)$.

\begin{lemma}\label{splitMethod}
    Let $\bK$ be an information pattern such that $\cD(\bA,\mathbb{I}_n,\bK, \mathbb{I}_n)$ satisfies Theorem~\ref{noStrucFixedModes}$-(b)$. In addition, let $\cC$ represent the disjoint union of cycles prescribed by Theorem~\ref{noStrucFixedModes}$-(b)$. Furthermore, note that $\cC$ can be partitioned into two sets $\cC_f$, corresponding to those comprising feedback links, and $\cC_s$, corresponding to those comprising only state variables.    Given a cycle $(u_{i_1},x_{i_1},x_{i_2},\dots,x_{i_k},y_{i_k},u_{i_1}) \in \cC_f$, then for any $l=1,\dots,k$, the information pattern $\bK^l$ such that $\bK^l_{i_1,i_k} = 0$, $\bK^l_{i_1,i_l} = \bK^l_{i_{l+1},i_l} = 1$ and $\bK^l_{r,t} = \bK_{r,t}$ for all other values of $r,t$, satisfies the condition of Theorem~\ref{noStrucFixedModes}--$(b)$.
    \mbox{~}\hfill$\diamond$
\end{lemma}

\begin{proof}
    Note that by setting $\bK^l_{i_1,i_k}=0$, we remove the cycle $\cC_1=(u_{i_1},x_{i_1},\dots,x_{i_k},y_{i_k},u_{i_1})$ from $\cC_f$. However, by setting $\bK^l_{i_1,i_l}=\bK^l_{i_{l+1},i_l}=1$ we add $\cC_2=(u_{i_1},x_{i_1},\dots,x_{i_l},y_{i_l},u_{i_1})$ and $\cC_3=(u_{i_{l+1}},x_{i_{l+1}},\dots,x_{i_k},y_{i_k},u_{i_{l+1}})$ to $\cC_f$. Further, note that, jointly, these two cycles cover all of the state vertices in $(u_{i_1},x_{i_1},\dots,x_{i_k},y_{i_k},u_{i_1})$, and thus $\cC_s\cup(\cC_f\setminus\{\cC_1\})\cup\{\cC_2,\cC_3\}$ is a set of disjoint cycles covering $\cD(\bA,\mathbb{I}_n,\bK^l,\mathbb{I}_n)$, and so $\bK^l$ satisfies Theorem~\ref{noStrucFixedModes}--$(b)$.
\end{proof}

\begin{remark}
    Note that Lemma~\ref{bisection} and Lemma~\ref{splitMethod} describe strategies  by which one can design feasible information  patterns given a feasible  information patterns. In particular, given an essential information pattern, we can obtain another essential feasible information pattern. Furthermore, both strategies presented can be efficiently computed, i.e., resorting to polynomially complexity (in the dimensions of the state, input and output) algorithms.
    \mbox{~}\hfill$\diamond$
\end{remark}

\vspace{-0.2cm}

\section{Illustrative examples}\label{illustrativeexample}

\vspace{-0.1cm}

In Figure~\ref{bisection3}, for a given digraph $\cD(\bA)$, we illustrate the use of Lemma~\ref{bisection} that can be interpreted as follows. Given an information pattern $\bK$ satisfying Theorem~\ref{noStrucFixedModes}--$(a)$, with   $\bK_{1,4}$ and two states $x_{3}$ and $x_{2}$  in the same SCC of $\cD(\bA)$, we  can remove the feedback edge $(y_4,u_1)$ and consider instead two feedback edges $(y_4,u_3)$ and $(y_2,u_1)$. The closed-loop digraph can be seen to have the same number of SCCs comprising state vertices, and  still satisfies Theorem~\ref{noStrucFixedModes}--$(a)$. Thus, provided that the starting feedback pattern was essential, this method allows us to determine other essential information  patterns that are not the sparsest.

\begin{figure}[htpb]
 \centering
  \includegraphics[scale=0.3]{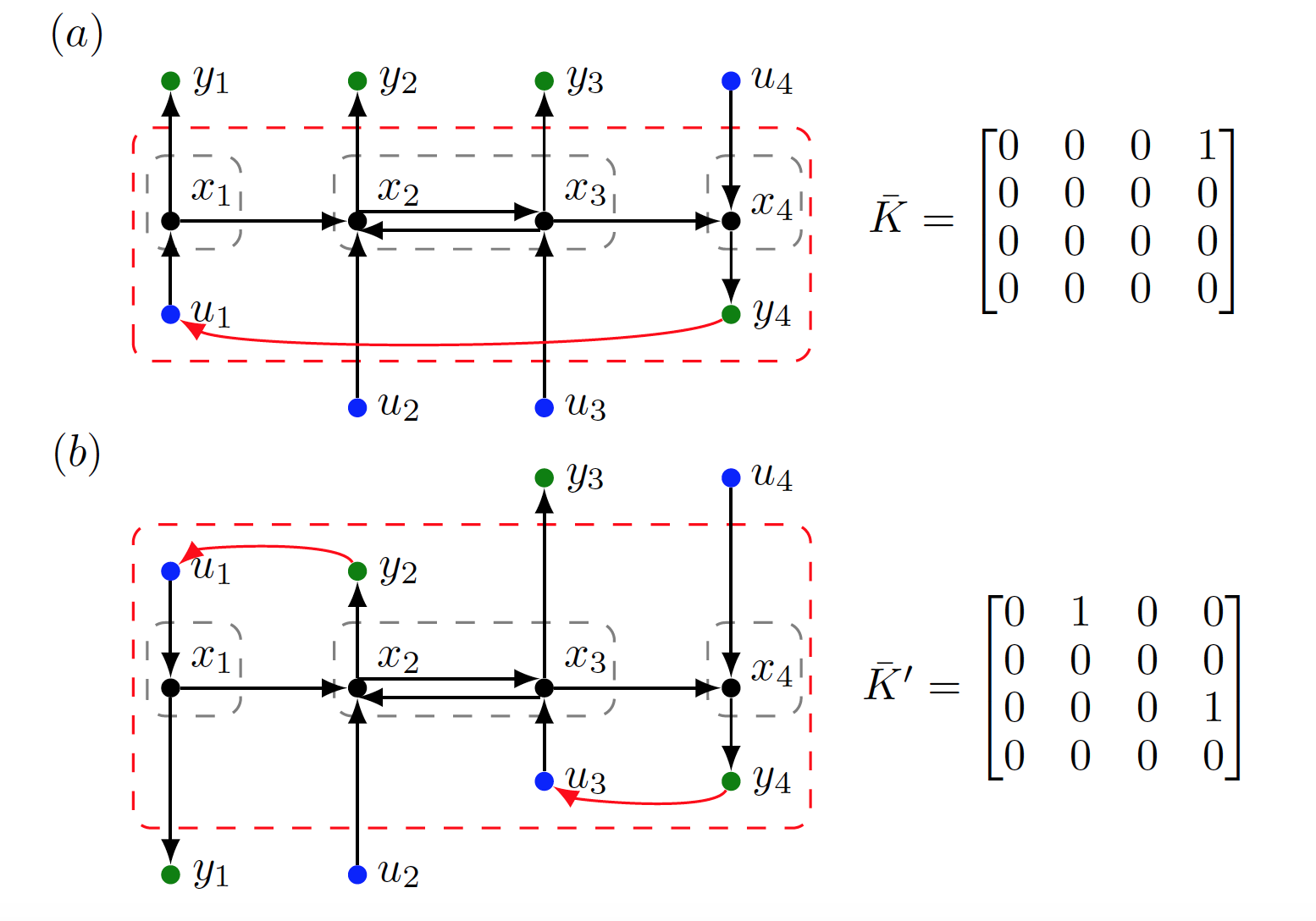}
    \caption[]{The SCCs of the state digraph $\cD(\bA)$ are depicted within grey dashed boxes, whereas the SCCs of the closed-loop digraph containing state vertices are depicted within the red dashed boxes. The state, input and output vertices are depicted in black, blue and green, respectively. In addition, the feedback links from the outputs to the inputs are depicted in red, and are related with the non-zero entries of the information patterns presented in the right-hand side. In $(a)$ we present the closed-loop digraph satisfying the assumptions in Lemma~\ref{bisection}, and in $(b)$ we present one possible conclusion from~it.}%
\label{bisection3}%
\end{figure}

\begin{figure}[htpb!]
 \centering
  \includegraphics[scale=0.3]{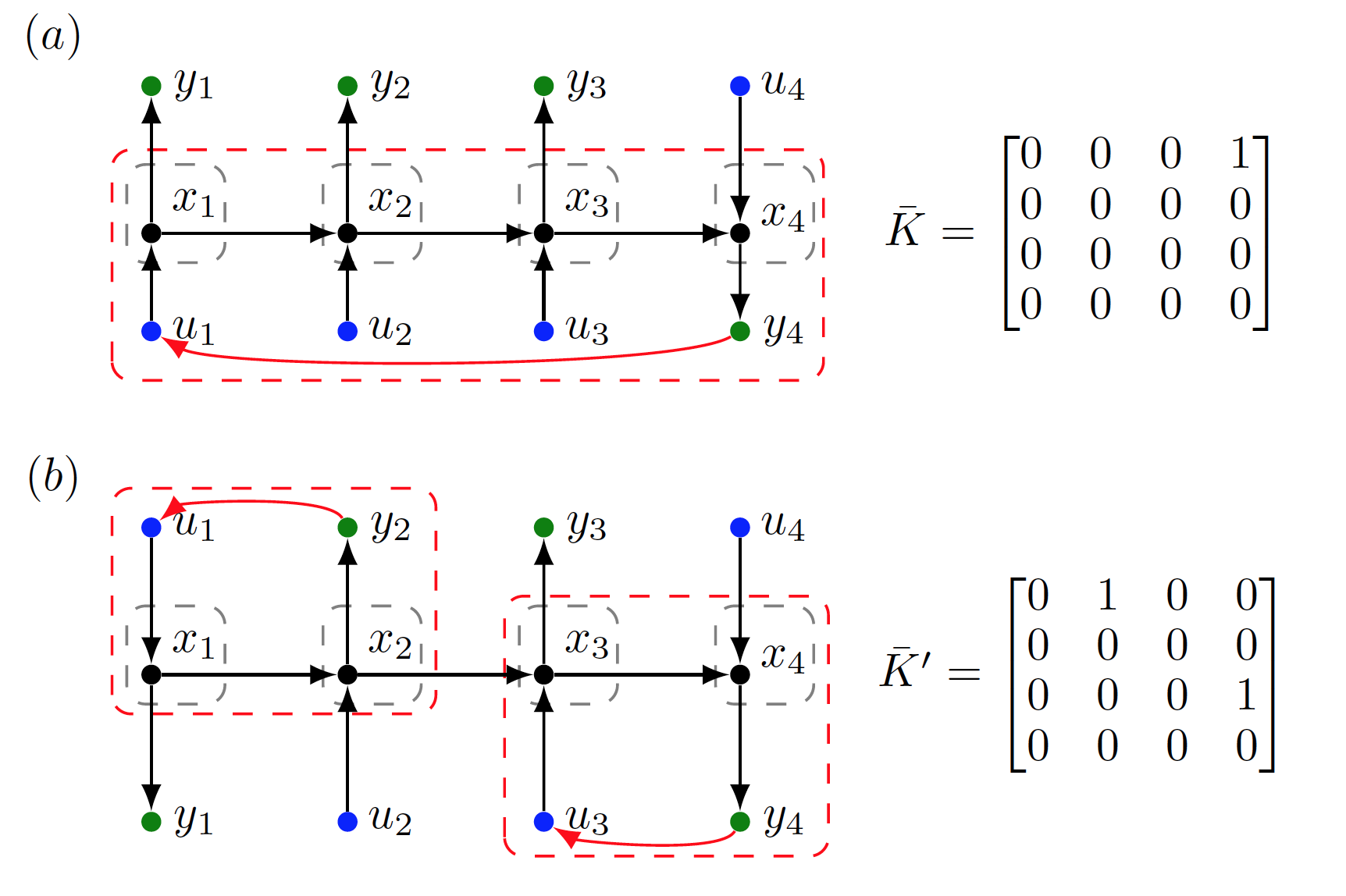}
    \caption[]{Adopting the same graphical representation explained in the caption of Figure~3, in $(a)$ we present the closed-loop digraph satisfying the assumptions in Lemma~\ref{splitMethod}, and in $(b)$ one possible conclusion from it.}%
\label{bisection4}%
\end{figure}

Similarly, in Figure~\ref{bisection4} we illustrate the use of Lemma~\ref{splitMethod} to be interpreted as follows. Once again, given $\bK$ satisfying Theorem~\ref{noStrucFixedModes}--$(b)$, with  $\bK_{1,4}=1$ and given two consecutive vertices $x_{2}$ and $ x_{3}$ in the cycle $\mathcal C$ comprising $(y_4,u_1)$ we can split $\mathcal C$ into two cycles $\mathcal C_1$ and  $\mathcal C_2$, where $\mathcal C_1$ comprises the edge $(y_{2},u_1)$ and $\mathcal C_2$ comprises $(y_4,u_{3})$. Once again, provided that the initial information pattern is essential and satisfies  Theorem~\ref{noStrucFixedModes}--$(b)$, (e.g.\ the sparsest, which can be obtained by Corollary~\ref{cor:informationCondb}) then a new essential information  pattern is formed satisfying  Theorem~\ref{noStrucFixedModes}--$(b)$.

\vspace{-0.1cm}

\section{Conclusions and further research}\label{conclusions}
\vspace{-0.1cm}

In this paper, we provided  several necessary and sufficient conditions to verify whether an information pattern is an essential pattern. In addition, we showed that the problem of determining essential feasible information patterns is NP-hard. Finally,  we provide a set of strategies that, given  an essential information pattern, enable us to determine a collection of essential information patterns.

The results provide new insights on schemes that can be used to approximate  minimum cost feasible information patterns, where arbitrary costs are attributed to the communication links. In addition, the characterization of the essential patterns provide insights on the resilience properties of closed-loop systems, with respect to changes in the information pattern. Both of  these problems will be studied as part of future research.
\vspace{-0.2cm}

\section*{Acknowledgements}
{\small
J. F. Carvalho thanks GRASP lab for their hospitality, as much of the writing of this paper and the research that led to it were carried out while visiting University of Pennsylvania at the beginning of Spring'15.
}

\bibliographystyle{IEEEtran}
\bibliography{IEEEabrv,bibliography}
\end{document}